\documentclass[11pt]{article}
\usepackage{todonotes}
\usepackage[a4paper,scale=0.8]{geometry}
\usepackage{amsmath}
\usepackage{amsthm}
\usepackage{amssymb}
\usepackage{bm}  
\usepackage{url}
\usepackage[numbers]{natbib}
\usepackage{enumerate}
\usepackage{graphicx}
\usepackage{blkarray}
\usepackage{color}
\usepackage{caption}
\usepackage{subcaption}
\usepackage{bbm}

\newtheorem{thm}{Theorem}

\newtheorem{lem}[thm]{Lemma}
\newtheorem{pro}[thm]{Proposition}

\newcommand{\fo}{\phi}
\newcommand{\conv}{\mathop{\mathbb{CH}}}

\newcommand{\one}{{\mathbbm{1}}}

\newcommand{\RR}{{\mathbb{R}}}

\title{Compact linear programs for 2SAT}
\author{David Avis$^{1,2}$ 
\thanks{Email: \texttt{avis@cs.mcgill.ca}}
\and Hans Raj Tiwary$^{3}$
  \thanks{Email: \texttt{hansraj@kam.mff.cuni.cz}}}

\let\cite\citep

\BAnewcolumntype{s}{>{$\small$}c}

\newlength{\normalparindent}
\newlength{\normalparskip}
\begin{document}
\maketitle
\footnotetext[1]{GERAD and School of Computer Science, McGill University,
  3480 University Street, Montr\'eal, Qu\'ebec, Canada H3A 2A7.}
\footnotetext[2]{Graduate School of Informatics,
  Kyoto University, Sakyo-ku, Yoshida Yoshida, Kyoto 606-8501, Japan.}
\footnotetext[3]{Department of Applied Mathematics (KAM) and Institute of Theoretical Computer Science (ITI), Charles University, Malostransk\'e n\'am. 25, 118 00 Prague 1, Czech Republic.}

\begin{abstract}
For each integer $n$ we 
present an explicit formulation of a
compact linear program, with $O(n^3)$ variables and constraints, which determines the satisfiability of any 2SAT formula with
$n$ boolean variables by a single linear optimization. This contrasts with the fact that the natural
polytope for this problem, formed from the convex hull of all satisfiable formulas and their
satisfying assignments, has
superpolynomial extension complexity. Our formulation is based on multicommodity flows.
We also discuss connections of these results to the stable matching problem.
\end{abstract}
\section{Introduction}
\label{intro}
Let $n$ be a positive integer and let $ x_1, ..., x_n, \bar{x_1},...,\bar{x_n}$ be a corresponding
set of $2n$ boolean valued literals. 
An input for the 2SAT problem is a formula $\fo$ representing the conjunctive normal form of a collection 
of $m$ 
clauses each of which contains two literals. The formula is satisfiable if an assignment
of truth values to $x_1,...,x_n$ makes at least one literal in each clause true.
The 2SAT problem can be solved optimally in $O(n+m)$ time by an algorithm due to 
Aspvall et al.~\cite{APT79} based on path-finding. 
Let $w_1,...,w_n$ be a set of $n$ integer weights.
The problem of finding a satisfying assignment of maximum weight
is known as the weighted 2SAT problem. It is NP-hard 
since it contains the maximum independent set problem for graphs as a special case.
In the main part of this note we are concerned with polyhedra related to the unweighted 2SAT problem.

Consider solving unweighted 2SAT for a formula $\fo$ by using a linear program
of the form $max~c^T y: Ay \le b$ of small size.
More formally we call such a formulation {\em compact} if:
(a) the dimensions of $A$ are polynomially bounded in $n$,
(b) there is an $O(log~n)$-space algorithm which constructs
an objective function $c=c_\fo$
from  formula $\fo$, and 
(c) there is an $O(log~n)$-space algorithm that takes the solution to
$max~c_\fo^T y: Ay \le b$ and decides whether or not $\fo$ is satisfiable.

In \cite{AT15} the authors showed that a natural LP formulation for 2SAT
has superpolynomial extension complexity. Recently this lower bound was improved
by G\"o\"os to $2^{\Omega(n/\log{n})}$ where $n$ is the size of the 2SAT formula \cite{Goos0016}. 
This implies that there is no
polyhedron of polynomial size that projects onto this polyhedron.
However, since the unweighted 2SAT problem is
in $P$ and linear programming is $P$-complete, compact LPs for it must exist. In fact, in
Avis et al. \cite{ABTW} a direct method is given to produce a compact LP from a polynomial
time algorithm. Applied to the algorithm of \cite{APT79} this produces an LP with
$O(n^4 log~n )$ constraints. The LP formulation is not explicit but is produced by a 
compiler\footnote{https://gitlab.com/sparktope/sparktope}
which transforms an implementation of the given algorithm
in a simple programming language. In the next section we give an explicit
compact LP for 2SAT with $O(n^3)$ constraints and variables which is based on multicommodity flows.
In Section \ref{remarks} we discussion connections with the stable matching problem 
and give some open problems.

\section{2SAT}
\label{2sat}
Let $\fo$ be an instance of 2SAT with $n$ variables.
We may assume that there
is no clause of the form $x_i \vee \bar{x_i}$, since this is always satisfied. 
Therefore there are at most $2n^2-n$ distinct clauses
which we may label, say lexicographically, $C_1,...,C_{2n^2-n}$.
The formula $\fo$ may be expressed by 
a 0/1 vector $y^\fo$ of length $2n^2-n$ by setting $y_i^\fo=1$ if and only if $C_i \in \fo$.
Let $x$ denote a 0/1 vector of length $n$ and $\one$ denote the vector of all ones, where the
length is defined by the context.
A natural polytope for 2SAT can be defined in $2n^2$ dimensions by
\begin{equation}
\label{Qn}
Q_n ~=~ \conv~\{(y^\fo,x) : \forall~x,~\fo~s.t. ~x{\it ~is~a~satisfying~assignment~of~\fo} \}
\end{equation}
which is the convex hull of all satisfiable formulae with their satisfying assignments. 

\begin{pro}
\label{ext}
Give a 2SAT formula $\fo$ we can solve either the unweighted or weighted satisfiability program
by solving a single LP over $Q_n$, which has superpolynomial extension complexity.
\end{pro}
\begin{proof}
Let $\fo$ be a 2SAT formula and first consider the unweighted satisfiability problem.
Define a vector $c$ of length $2n^2-n$ by setting $c_i=1$ if $y_i^\fo=1$ and $c_i= -1$ otherwise.
Consider $z^*=\{max~c^T y : (y,x) \in Q_n \}$ and let $(y^*,x^*)$ be an optimal solution. It is easy to verify that 
$z^* = \one^T y^\fo$ if $\fo$ is satisfiable and
$z^* \le \one^T y^\fo~-1$
otherwise.

For the weighted satisfiability problem,
let $w=(w_1,...,w_n)$ be the vector of given integer weights
and set $W=\sum_{i=1}^n |w_i|$.
Consider $z^*=\{max~ c^T y + \frac{1}{3W} w^T x : (y,x) \in Q_n \}$.
If $\fo$ is unsatisfiable, then 
$z^* \le \one^T y^\fo -1 + \frac{W}{3W} = \one^T y^\fo -\frac{2}{3}$.
Otherwise $z^* \ge \one^T y^\fo - \frac{W}{3W}  = \one^T y^\fo -\frac{1}{3}$.
So again inspecting the value of $z^*$ determines the satisfiability of $\fo$.
Let $(y^*,x^*)$ be an optimum solution.
When $\fo$ is satisfiable we must have $c^T y^* = \one^T y^\fo$  and $x^*$
must be a satisfying assignment of $\fo$. 
So the values of $x_i^*$ at optimality will 
give a maximum weight satisfying assignment.

Next we consider the extension complexity of $Q_n$. 
First observe that if we fix a 2SAT formula $\fo$ and restrict ourselves
to the face where the first $2n^2-n$ coordinates are equal to $y^\fo$ we obtain a face of $Q_n$
that contains all satisfying assignments for $\fo$.
It was shown in \cite{AT15} that 
there is a family of 2SAT formulas $\fo_n$,
indexed by the number of variables $n$, 
for which the set of satisfying assignments 
has superpolynomial extension complexity. 
Since the satisfying assignments for the formula $\fo_n$
correspond to a face of $Q_n$ this face has superpolynomial extension complexity. 
It follows that $Q_n$ does also.
\end{proof}
The proposition implies that a direct approach to solving 2SAT as a linear program
over $Q_n$
will require exponentially many constraints. In what follows we give a completely different 
construction. It follows from the proposition that there will be no projection
of this polytope onto $Q_n$.

In the sequel we define variables $x_{n+i}$ to denote the literals $\bar{x}_i$, for 
$i=1,...,n$. 
Subscripts should be taken $mod~2n$ and we usually omit
this for clarity.
For a given $n \ge 2$ consider the complete bidirected graph on $2n$ vertices labeled
$x_1,\ldots,x_{2n}$. 
We add vertices $t_1,\ldots,t_{2n}$ and edges $(t_i,x_i),(x_i,t_i),(x_{n+i},t_{n+i}),$
and $(t_{n+i},x_{n+i})$ to the graph to obtain a graph with $4n$ vertices and 
$2n(2n-1)+4n=4n^2+2n$ edges. Call this graph $D_n$. 
We define a multicommodity flow problem with $2n$ commodities on $D_n$ as follows.
For each $i=1,...,2n$,
there are a pair of terminals
$t_i$ and $t_{n+i}$ and a distinct flow, labeled flow $i$, between them.
The capacity on edge $(t_i, x_i)$ is one and all other capacities are infinite.
Therefore a unit flow from $t_i$ to $t_{n+i}$ saturates only the initial edge and corresponds
to a path in $D_n$ between these two terminals. 
The sum of the maximum flows over all $2n$ commodities must therefore be at most $2n$
and this will be achieved if and only if there is a path between each of the $2n$
pairs of terminals.
Now consider the following
multicommodity flow polytope where $f_{ab}^k$ denotes the value of flow
$k$ on edge $ab$ of $D_n$ 
and $\delta_{ij}=1$ if $i=j$ and zero otherwise.

\[\begin{array}{rclr}
\delta_{i,k}f^i_{t_ix_i}+\displaystyle\sum_{i\neq j}f^k_{x_jx_i}&=&
\displaystyle\sum_{i\neq j}f^k_{x_ix_j}+\delta_{n+i,k}f^{n+i}_{t_{n+i}x_{n+i}}&i,k=1,\cdots,2n\\
P_n:~~~~~~~~~~~~~~~~~~~~~~~~0~\leqslant~f^i_{t_ix_i}&\leqslant&1&i=1,\cdots,2n\\
f^k_{x_ix_j}&\geqslant&0&~i,j,k=1,\cdots,2n,\\&&& i \neq j
\end{array}\]

Let $f$ denote the ensemble of $2n \cdot 2n(2n-1)+4n=8n^3-4n^2+4n$ variables
$\{f^i_{t_ix_i}, f^k_{x_ix_j} \}$ ordered in any consistent way.
The polytope $P_n$ has $8n^3-4n^2+6n$
inequalities and $4n^2$ equations. 
Note that the constraints consist of $2n$ {\em disjoint} copies of single commodity
flow constraints and thus are
totally unimodular.
So in view of the upper and lower bound constraints the polytope has 0/1 vertices.

Let $\fo$ be a 2-SAT formula and suppose it contains the clause $x_i\lor x_j$ for some
$i$ and $j$.
This means that $\fo$ forces 
implications $\overline{x}_i\Rightarrow x_j$ and $\overline{x}_j\Rightarrow x_i$. 
In the graph $D_n$ these implications correspond to edges $x_{n+i},x_j$ and
$x_{n+j},x_i$ respectively. 
For each such clause that is missing from $\fo$ these implications are missing also.
We construct an {\em implication graph} $G_n$ from $D_n$ for $\fo$ by removing the
corresponding edges for all missing clauses.
In the flow setting, if clause $x_i\lor x_j$ is missing we
set the variables $f^k_{x_{n+i}x_j}=f^k_{x_{n+j}x_i}=0$, 
$k=1,...,2n$.
By so doing for each clause missing from $\fo$ we
restrict ourselves to a face of $P_n$ which we denote
$P^{\fo}_n$.
A {\em maximum commodity flow} is a feasible vector in $P_n$ that maximizes $\sum_{i=1}^{2n} f^i_{t_ix_i}$.
\begin{lem}
\label{max}
Any maximum commodity flow lying on the face $P^{\fo}_n$ has $f^i_{t_ix_i}=f^{n+i}_{x_{n+i}t_{n+i}}=1$
for some $i$ if and only if the formula $\fo$ is unsatisfiable.
\end{lem}
\begin{proof}
If $f^i_{t_ix_i}=f^{n+i}_{x_{n+i}t_{n+i}}=1$, for some $i$, then these two flows imply paths
from $x_i$ to $x_{n+i}$ and back. The first path gives a chain of implications implying that
if $x_i$ is true then $x_{n+i}=\bar{x_i}$ is true, a contradiction. Similarly the
second path gives implications showing that if $x_i$ false then it is
true. 
Therefore the formula $\fo$ is
unsatisfiable.

Conversely, if $\fo$ is unsatisfiable then there is a path from $x_i$ to $x_{n+i}$
and a path from $x_{n+i}$
to $x_{i}$.
Thus a unit flow can be sent from $t_i$ to 
$t_{n+i}$ and from $t_{n+i}$
to $t_i$ and $f^i_{t_ix_i}=f^{n+i}_{x_{n+i}t_{n+i}}=1$.
Since these two paths cannot use any edges
corresponding to missing clauses in $\fo$, the corresponding flows lie on the face $P^{\fo}_n$. 
\end{proof}

Testing the satisfiability of $\fo$ can be achieved
by maximizing a suitably
chosen linear objective function over $P_n$ designed so that all maximum solutions
are found on the
face $P^{\fo}_n$. 

\begin{thm}
\label{thm:2sat}
Let $\fo$ be a 2-SAT formula with $n$ variables. Consider the following LP:

\[z^*~=~\max \displaystyle\sum_{i=1}^{2n} f^i_{t_ix_i}
-(2n+1)\sum_{\substack{(x_i\lor x_j)\notin \fo \\1\leqslant k\leqslant 2n}}(f^k_{x_{n+i}x_j}+f^k_{x_{n+j}x_i})\\
\]
\[
f \in P_n
\]
Then, $\fo$ is unsatisfiable
if and only if at optimality there is an $i$ such that
$f^i_{t_ix_i}=f^{n+i}_{x_{n+i}t_{n+i}}=1$.
A certificate for the satisfiability/unsatisfiability of $\fo$ can
be obtained from the optimum solution.
\end{thm}
\begin{proof}
The LP is feasible, since this can be achieved by setting all variables to zero, and the objective function 
is upper bounded 
by $2n$. Therefore optimum solutions exist and $0 \le z^* \le 2n$. We show that they lie on the face $P^{\fo}_n$.
Consider any extreme point of $P_n$ that provides an optimum solution for the LP. As remarked above this is integral and is in fact
0/1 valued since $f^i_{t_ix_i} \le 1,  i=1,...,2n$.
Suppose that $(x_i\lor x_j)$ is a clause missing from $\fo$. 
For $k=1,...,2n$, if 
 either of the variables $f^k_{x_{n+i}x_j}$ or $ f^k_{x_{n+j}x_i}$ is set to one then this incurs a 
penalty of $-(2n+1)$ in the objective function. The rest of the objective function can only contribute $2n$ 
at most and so $z^* < 0$, a contradiction.
Therefore all optimum vertices, and hence all optimum solutions, occur on the face $P^{\fo}_n$.
It now follows from Lemma \ref{max} that the satisfiablility/unsatisfiability
of $\fo$ can be determined from the optimum solution. The lemma gives
a certificate of unsatisfiability. For a certificate of satisfiability we
construct a satisfying assignment. For each $i$ such that
 $f^i_{t_ix_i} = 1$ set $x_{n+i}=1$ and $x_i=0$. For each edge $x_i, x_j$
with $x_i=1$ and $x_j$ unlabelled, set $x_j=1$ and $x_{n+j}=0$.
Continuing in this
way we label a set of strong components in $G_n$.
If at termination there remain any 
unlabelled vertices $x_i$ we choose
one arbitrarily, set $x_i=1$, $x_{n+i}=0$ and continue as before.
Since no strong component can contain both $x_i$ and $x_{n+i}$ we have
consistently labelled
$G_n$ and have a satisfying assignment.
\end{proof}

Both the construction of the
objective function and determining
the (un)satisfiability of $\fo$ from the optimum solution can be
done with $O(\log n)$ space. So the LP described in Theorem \ref{thm:2sat} is compact.
As remarked earlier in this section this implies that it cannot be linearly projected onto $Q_n$
and, in particular, the satisfying assignment $x_i$ cannot be directly read from the optimum solution.
Furthermore it apparently gives no insight into solving the weighted 2SAT problem.

\section{Connections to stable matchings and open problems}
\label{remarks}

From the point of view of extension complexity, 2SAT is 
an interesting problem
since the unweighted version is in $P$ and the weighted version
is NP-hard. Although the natural polyhedral formulation of 2SAT 
has high extension complexity we have shown that a compact explicit polyhedral formulation
for the unweighted version exists. 
A problem of similar type is the stable matching problem, which has 
many variations and has received
considerable attention since the seminal Gale-Shapley paper in 1962.
For recent developments see Cseh and Manlove \cite{CM16}.

In the stable matching problem we have a group of $2n$
people each of who
has a list ranking the other $2n-1$ people in order of preference. 
A {\em matching} is a set of $n$ pairs of people such that each person is in exactly
one pair. An {\em blocking pair} is a set of 4 people $i,j,k,l$ such that $i$ is matched
to $k$ but prefers $j$, and $j$ is matched to $l$ but prefers $i$.
A matching is {\em stable} if it contains no blocking pair.
A considerable amount of research has been done on this problem
and computationally the 2SAT problem and stable matchings are closely related.
Like 2SAT,
there is a polynomial time algorithm for the problem (Irving~\cite{Ir85})
and the weighted version is NP hard (Feder~\cite{Fe92}).

From the polyhedral viewpoint we can define a polytope $R_{2n}$ for the
stable matching which is analogous to $Q_n$ for 2SAT.
Each instance $\psi$ for the stable matching problem consists of $2n$ permutations of $2n-1$ elements,
each of which can be encoded as a $(2n-1)$ by $(2n-1)$ permutation matrix.
Let the corresponding binary vector of length $2n(2n-1)^2$ be denoted $y^{\psi}$. A matching for
$\psi$ can be encoded as binary vector $x$ of length $(2n-1)n$
where $x_{ij}=1$ if and only if $i$ and $j$ are matched, for $1 \le i < j \le 2n$. 
We define $R_{2n} \subset \RR^{n(2n-1)(4n-1)}$ by
\begin{equation}
\label{R2n}
R_{2n} ~=~ \conv~\{(y^{\psi},x) : \forall~x,~\psi~s.t. ~x{\it ~is~a~stable~matching~for~\psi} \}
\end{equation}
which is the convex hull of all stable matching problems with their stable matchings.
Note that in the above we only consider instances $\psi$ which have stable matchings.
Since the weighted stable matching problem can be solved by an LP
over $R_{2n}$ we expect that it has high extension complexity, a fact we now prove.

\begin{pro}
$R_{2n}$ has superpolynomial extension complexity.
\end{pro}
\begin{proof}
Theorem 8.2 of Feder~\cite{Fe92} shows that a 2SAT formula with $n$ variables
can be transformed to a stable matching problem with $4n$ people
so that the set of solutions of the 2SAT problem correspond to the set of solutions of the
stable matching problem. Proceeding as in Proposition \ref{ext} we make use
of the family of 2SAT formulas $\phi_n$ from \cite{AT15} that yield
a family of faces of 2SAT polytopes that have superpolynomial
extension complexity. Applying Feder's transformation, each $\phi_n$ maps to an instance
$\psi_{4n}$ 
of the stable matching problem whose stable matchings correspond to the
satisfying assignments of $\phi_n$. If we now restrict ourselves to the face where the 
first $2n(2n-1)^2$ coordinates are $y^{\psi_{4n}}$ we obtain a face with all stable matchings 
of the instance $\psi_{4n}$. This is a face of $R_{4n}$ with superpolynomial extension complexity,
and the proposition follows.
\end{proof}

It is natural to ask if, like the 2SAT problem, there is an explicit compact LP formulation 
for the unweighted problem.
Indeed it is known that a stable matching problem with $2n$ people can be
formulated as a 2SAT problem with $O(n^2)$ variables and $O(n^2)$ clauses.
Algorithms for finding this formulation were given by Gusfield \cite{Gu88},
and Feder \cite{Fe94} whose algorithm requires $O(n^2)$ time and space.
At first glance it may appear that we can use our compact formulation for
2SAT via this transformation of stable matching problems to 2SAT problems.
Unfortunately this transformation cannot be performed using $O(log~n)$ space
so condition (b) of a compact formulation, as defined in the Introduction, fails.
So for the moment we must fall back on the method of \cite{ABTW} to 
obtain a compact formulation with $O(n^4~log~n)$ constraints.

A final open problem of this type
concerns the perfect matching problem which Rothvo{\ss} proved also has superpolynomial
extension complexity~\cite{Ro14}. The method of \cite{ABTW} 
applied to the $O(n^{2.5})$ time and $O(n^2)$ space
implementation of Edmonds' algorithm by Micali and Vazirani~\cite{MV80} 
yields an LP that has $O(n^{4.5}~log~n)$ constraints.
It is therefore of great interest to see if this bound can be significantly improved,
and if an explicit compact LP formulation for the perfect matching problem exists.
\section*{Acknowledgments}

Research of the first author is supported by the JSPS under
a Kakenhi grant and a Grant-in-Aid for Scientific Research
on Innovative Areas, `Exploring the Limits of Computation (ELC)'. 
The second author was supported by project GA15-11559S of GA \v{C}R.

\bibliographystyle{plain}
\bibliography{2sat}

\end{document}